\documentclass{amsart}  
\pdfoutput=1

\usepackage{graphicx}
\usepackage{latexsym}
\usepackage{amsfonts,amsmath,amssymb}

\newtheorem{lemma}{Lemma}

\newtheorem{theorem}{Theorem}
\newtheorem{proposition}{Proposition}

\newtheorem{definition}{Definition}
\def\sym{\operatorname{sym}}
\newcommand{\Var}{\operatorname{\mathbb{V}}}
\newcommand{\E}{\operatorname{\mathbb{E}}}
\def\ds{\displaystyle}
\def\wh{\widehat}
\def\cal{\mathcal}
\def\R{\mathbb{R}}

\def\Pr{\operatorname{\mathbb{P}}}

\newcommand{\Img}[2]{\includegraphics[width=#1truecm]{#2}}

\begin{document}

\author[M. B\'ona and P. Flajolet]{Mikl\'os B\'ona and Philippe Flajolet}
\title[Random Phylogenetic Trees]{Isomorphism and Symmetries in Random Phylogenetic Trees}
\address{\rm M. B\'ona, Department of Mathematics, 
University of Florida,
358 Little Hall, 
PO Box 118105, 
Gainesville, FL 32611--8105 (USA)
}
\address{\rm P. Flajolet. {\sc Algorithms} Project,
INRIA Rocquencourt, 
F-78153 Le Chesnay (France)
}
\date{January 6, 2009}

\begin{abstract} 
The  probability  that  two randomly selected   phylogenetic trees of the same size are
isomorphic is  found to be asymptotic to a decreasing exponential modulated by a 
polynomial factor. The number of symmetrical nodes in a random phylogenetic tree of
large size obeys a limiting Gaussian distribution, in the sense
of both central and local limits. The probability that two
random phylogenetic trees have the same number of symmetries asymptotically obeys an
inverse square-root law.
Precise estimates for these  problems are obtained   by methods of analytic combinatorics,
involving   bivariate generating
functions, singularity analysis, and quasi-powers approximations.
\end{abstract}

\vspace*{-2.25truecm}
\maketitle

\vspace*{-0.5truecm}

\section{Introduction}
Every high school student of every civilized part of the world is cognizant of
the \emph{tree of species}, also known as the ``tree of life'',
in relation to Darwin's theory of evolution (Figure~\ref{darwin-fig}). 
We observe $n$ different species,
and form a group with the closest pair (under some suitable proximity criterion), 
then repeat the process
with the $n-2$ remaining species together with the newly formed group, and so on.
In this way a \emph{phylogenetic tree},
also known as ``cladogram'', is obtained: such a tree 
has the $n$ species at its external nodes, also called ``leaves''; it has $n-1$ internal binary nodes,
% whose descendants are either internal or external nodes, 
and it is naturally rooted at
the last node obtained by the process. 
Note that, by design, there is no specified order between the two children of a binary node.

Seen from combinatorics, the phylogenetic trees under consideration are thus trees in the
usual sense of graph theory (i.e., acyclic connected graphs~\cite[\S1.5]{Diestel00});
in addition, a binary node is distinguished as the root,
and each node has outdegree either~0
(leaf) or~2 (internal binary node).
Finally, the % (distinguishable) 
leaves are labeled by distinct integers, which
we may  canonically take to be an integer interval $[1,n]$.
In classical combinatorial terms, the set of phylogenetic trees thus
 corresponds to
the set $\cal B$ of \emph{rooted non-plane binary trees}, which are 
\emph{labeled at their leaves}.

\begin{figure}\small
\Img{4.5}{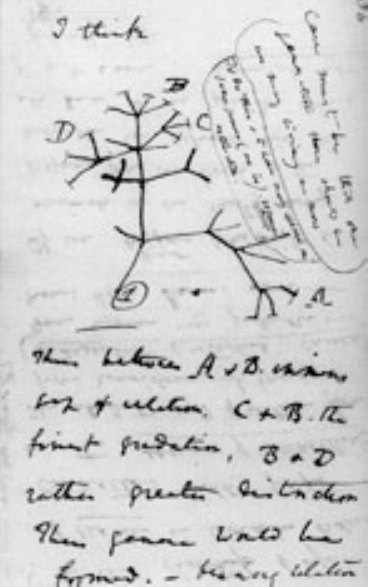}\qquad \Img{5}{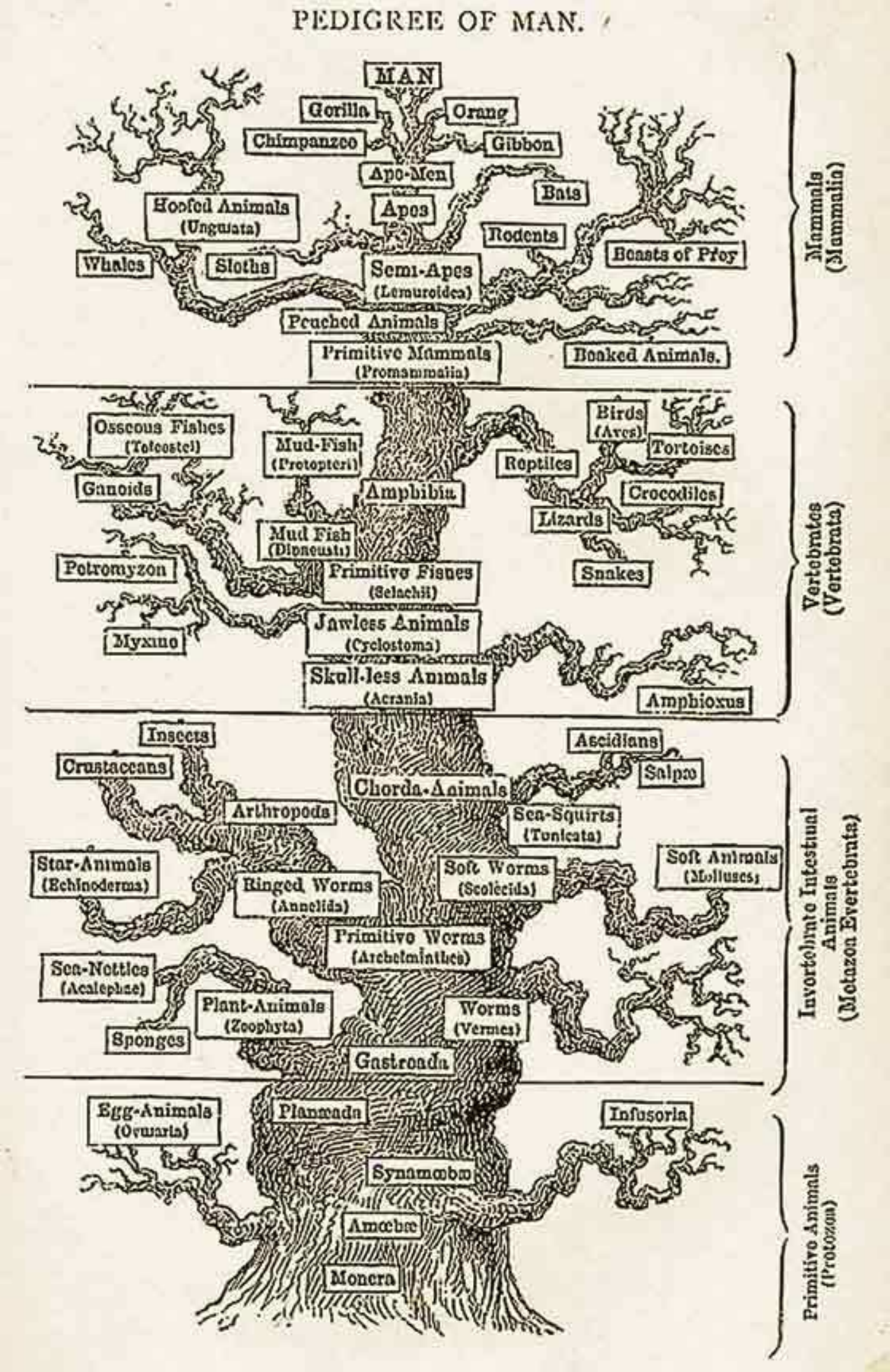} 
\caption{\label{darwin-fig}\small
{Left}: the representation of a pylogenetic tree in Darwin's own handwriting. 
{Right}: 
an illustration of the Tree of Life by Haeckel in 
\emph{The Evolution of Man}, published in 1879.
(Source: Entry ``Tree of life'', \emph{Wikipedia}.)
}
\end{figure}

We let $\cal B_n$ be the subset of~$\cal B$
corresponding to trees of size~$n$ (those with $n$ leaves)
and denote by~$b_n:=| \cal B_n|$ the corresponding cardinality.
Considering 
the listing of all unlabeled trees of sizes $1,2,3,4$ 
\begin{equation}\label{utrees14}
\begin{array}{c}\hbox{\setlength{\unitlength}{1truecm}
\begin{picture}(8,1.6)
\put(0,0){\Img{8}{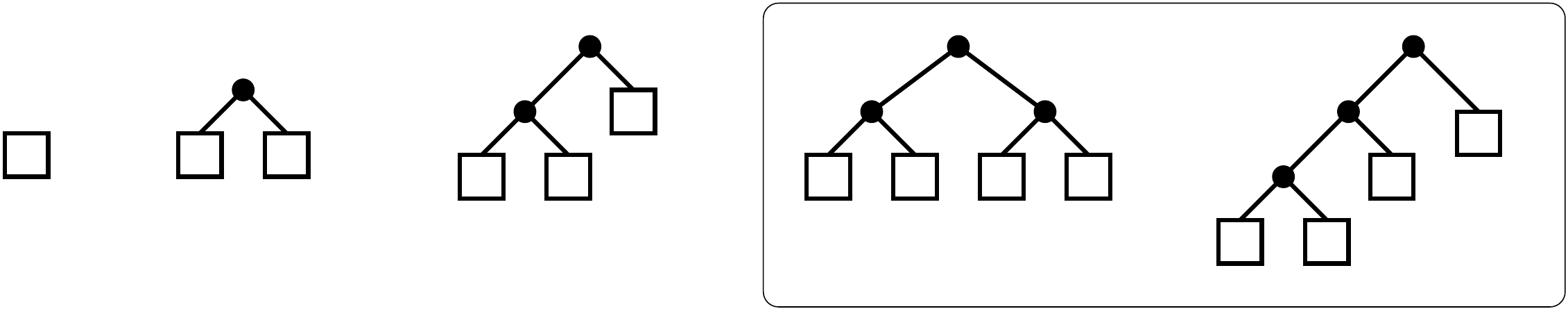}}
\put(4.7,0.20){$(L)$}
\put(7.3,0.20){$(R)$}
\end{picture}},\end{array}
\end{equation}
the reader is invited to verify that $b_1=1$, $b_2=1$, $b_3=3$,
and that $b_4=15$ is obtained by 
counting all possible labelings (3 and 12, respectively) of the two trees~$L,R$
shown on the right of~\eqref{utrees14}.

% \begin{figure}[ht]
%  \begin{center}
% %  \epsfig{file=u2trees.eps}
% 
% \caption{  \label{u2trees}The two elements of $U_4$.}
%  \end{center}
% \end{figure}

A general formula for the numbers $b_n$ is well known and
straightforward to
prove.  Indeed, if we introduce the \emph{exponential generating function}
\[
B(z):=\sum_{n\geq 1} b_n\frac{z^n}{n!},
\]
then
the fact that each element of $\cal B_n$ is built up from its two subtrees
implies that
\begin{equation} \label{labelfunc}
B(z)=z+ \frac{1}{2}{B(z)^2}.
\end{equation}
See the books by Stanley~\cite[pp.~13--15]{Stanley99} 
or Flajolet--Sedgewick~\cite[\S2.5]{FlSe08}
for details and related results.
So, $B(z)$ is the solution of the quadratic equation (\ref{labelfunc}) that is
a generating function. That is, 
\[
B(z)=1-\sqrt{1-2z}.\] 
This leads to the following  exact formula for the numbers $b_n$. 
  
\begin{proposition} \label{labelenum}
The number of phylogenetic trees on $n$ labeled
nodes is 
\[b_n=1\cdot 3\cdot \cdots (2n-3) \equiv (2n-3)!!.\]
\end{proposition}

% For instance, $b_4=15$, and indeed, one checks easily that the leaves of
%  tree $L$ in Figure \ref{u2trees} can be labeled in three
% different ways, whereas the leaves of  tree $R$ in 
%  Figure \ref{u2trees} can be labeled in 12 different ways. 

There is a natural way to associate an {\em unlabeled} rooted binary non-plane
tree to each element $t\in \cal B_n$, by simply removing all the labels of
$t$. We will say that two elements $t,t'\in \cal B_n$ are \emph{isomorphic} 
if removing their labels
will associate them to the same unlabeled tree. 
This leads to the following intriguing question. 

% \begin{question} 
\begin{quote}\em
{\bf Question.} What is the probability $p_n$ that two  
phylogenetic trees, selected uniformly at random in~$\cal B_n$,
are isomorphic?
\end{quote}
% \end{question}
\noindent
Note that, in our running example, the case
of $n=4$, we have
 $p_4=\left(\frac{1}{5}\right)^2+\left(\frac{4}{5}\right)^2=\frac{17}{25}$.
Indeed, if we selected two elements of $\cal B_4$ at random, 
there is a $(3/15)^2
=(1/5)^2$ chance that they will both belong to the isomorphism class
of $L$, and $(12/15)^2=(4/5)^2$ that they both belong to the isomorphism
class of $R$, where $L$ and $R$ are the two trees  of~\eqref{utrees14}.
% in Figure \ref{u2trees}. 

In this paper,
 we will use a multivariate generating function argument (Section~\ref{alg-sec}) in conjunction
with an analysis of singularities in the complex plane (Section~\ref{ana-sec})
to answer the isomorphism  question in Theorem~\ref{thm1}.
In Section~\ref{auto-sec},
we will extend our analysis to distributional  estimates of the number of
symmetrical nodes in phylogenetic trees and in their unlabeled counterparts,
 known as Otter trees: see Theorems~\ref{thm2} and~\ref{thm3} for 
\emph{central and local limit laws}, respectively.
Such results in particular quantify the distribution of the log-size of the automorphism group of 
the random trees under consideration. 
In Section~\ref{same-sec}, we will work out an explicit estimate of the probability that
two random trees have the same number of symmetries.

\section{Isomorphism: a Generating Function Argument}   \label{alg-sec}

\subsection{Unlabeled Trees}\label{unlabeled}
Let $\cal U_n$ be the set  of all \emph{unlabeled} rooted binary non-plane
trees  with  $n$  leaves, and  let $u_n=|\cal U_n|$ be the corresponding count,
with \emph{ordinary generating function}
\[
U(z):=\sum_{n\ge1} u_n z^n.
\]
Such  trees are often
called {\em Otter  trees}, since Otter  was  the first to study  their
enumeration~\cite{Otter48}.  We can build a  generic element of  $\cal U_n$
by taking a tree  $t'\in \cal U_k$ and a tree  $t''\in \cal U_{n-k}$, and
 joining
their   roots to a  new   root.  As  the   order of $t$   and  $t'$ is
not significant, we  get each  tree  $t\in \cal U_n$   {\em twice}  this way,
except that, if the two subtrees of $t$ are identical,  we get $t$
only     once.         This       leads    to      the      functional
equation~\cite{FlSe08,HaPa73,Otter48,PoRe87}:
\begin{equation} \label{otterfunc}
U(z)=z+\frac{1}{2} \left( U(z)^2 + U(z^2) \right).
\end{equation}
The numbers $u_n$ 
are listed as sequence A001190 
(the \emph{``Wedderburn--Etherington numbers''}) in the On-line Encyclopedia
 of Integer
Sequences by Neil Sloane~\cite{Sloane08} and are the answers to various 
combinatorial enumeration
 problems.
The first few values of the sequence $\{u_n\}_{n\geq 1}$ are 
1, 1, 1, 2, 3, 6, 11, 23, 46, 98.

\subsection{A multivariate generating function}

Let $t_1\in \cal B_n$, and let $t_2\in \cal B_n$. By Proposition 
\ref{labelenum}, 
there are $(2n-3)!!^2$ possibilities for the ordered pair $(t_1,t_2)$,
where $t_1$ and $t_2$ do not have to be distinct.
Our goal is to
count such ordered pairs in which $t_1$ and $t_2$ are isomorphic. This
number, divided by $(2n-3)!!^2$ will then provide the probability $p_n$
that two randomly selected elements of $\cal B_n$ are isomorphic.

Let $t\in \cal U_n$.
Then the number of different labelings of the leaves of $t$ is 
\begin{equation}\label{wdef}
w(t)=\frac{n!}{2^{\sym(t)}}, 
\end{equation}
where $\sym(t)$ is the number of non-leaf nodes $v$ of $t$ such that
the two subtrees stemming from~$v$ are identical. For example, if $n=4$, and
$t$ is the tree $L$ of~\eqref{utrees14}, then we have
$w(t)=3$, and indeed, $t$ has $n!/2^3=24/8=3$ labelings. If $t$ is the 
tree $R$ of~\eqref{utrees14}, then we have $w(t)=1$, and $t$ has
$24/2=12$ labelings. 

Isomorphism classes within $\cal B_n$ correspond to elements of $\cal U_n$.
Set \begin{equation}
\label{normalize}
 W_n=\sum_{t\in \cal U_n} \frac{1}{2^{\sym(t)}}.\end{equation}
 As we have mentioned above, ${n!}/{2^{\sym(t)}}$ is the number of 
labeled trees in the isomorphism class corresponding to $t$.
 Summing this number over
all isomorphism classes, we obtain
the total number  of trees in $\cal B_n$. That is, 
 \[n!W_n=1\cdot 3\cdot \cdots (2n-3)!!.\]
For instance, $W_4=\frac{1}{8} + \frac{1}{2}=\frac{5}{8}$, and
$4!\cdot  \frac{5}{8}=15=5!!$.

Let 
\begin{equation}\label{deff}
F(z,u)=\sum_{t\in \cal U} u^{\sym(t)}z^{|t|}
\end{equation}
be the bivariate 
generating function of Otter trees, with $z$ marking the number of leaves,
and $u$ marking non-leaf nodes with two identical subtrees. 
In particular, $F(z,u)=z+uz^2+uz^3+(u^3+u)z^4+\hbox{higher degree terms}$.
The crucial observation about $F(z,u)$ is the following.

\begin{lemma} \label{lemmafunc}
The bivariate generating function $F(z,u)$ that enumerates Otter trees with 
respect to the number of symmetrical nodes satisfies the functional equation
\begin{equation} \label{funceq}
F(z,u)=z+\frac{1}{2}F(z,u)^2+\left(u-\frac{1}{2} \right)
F(z^2,u^2). \end{equation}
\end{lemma}

\begin{proof}
If a tree consists of more than one node, then it is built 
up from its two subtrees. As the order of the two subtrees is not significant,
we will get each tree {\em twice} this way, except the trees whose two 
subtrees are identical. If $t_1$ and $t_2$ are the two subtrees of 
$t$ whose roots are the two children of the root of $t$, then 
\[\sym(t)=\left\{ 
% \begin{array}{l@{\ }l}
\begin{array}{ll}
\sym(t_1)+ \sym(t_2), & \hbox{  if $t_1$ and $t_2$ are not identical}\\
\sym(t_1)+ \sym(t_2) +1, & \hbox{ if $t_1$ and $t_2$ are identical}.
\end{array}\right.
\]

The first term of the right-hand side of~\eqref{funceq} represents the 
tree on one node,
the second term represents all other trees as explained in the preceding
paragraph, and the third term is the correction term for trees in which the
two subtrees of the root are identical.
\end{proof}

Note that various specializations of $F(z,u)$ have a known 
combinatorial meaning. Indeed,
\begin{itemize}
\item[$(i)$] If $u=1$, then 
$F(z,1)=\sum_{t\in  U} z^{|t|}$ is simply the ordinary generating function~$U(z)$ of 
Otter trees with respect to their number of leaves. We have discussed
this generating function in Subsection \ref{unlabeled}, and mentioned
that its coefficients $u_n$ are the Wedderburn--Etherington numbers, 
which form sequence A001190 in \cite{Sloane08}.
\item[$(ii)$] If $u=2$, then $F(z,2)
=\sum_{t\in  U} z^{|t|}2^{\sym(t)}$ is the ordinary generating
function of the total number of automorphisms in all Otter trees. 
The coefficients constitute sequence A003609 in \cite{Sloane08}. Interested readers may
consult McKeon's studies~\cite{McKeon91,McKeon96} for details. 
The first few elements of the
sequence are 1, 2, 2, 10, 14, 42, 90, 354. 
\item[$(iii)$] If $u=1/2$, then  \[F\left(z,\frac12\right)
=\sum_{t\in \cal  U} z^{|t|}2^{-\sym(t)}=\sum_n W_nz^n=
\sum_n (2n-3)!!\frac{z^n}{n!},\]
is the exponential
 generating function $B(z)$ of labeled trees in disguise.
We have discussed this generating function in the Introduction.
The numbers $(2n-3)!!$ form sequence A001147 in \cite{Sloane08}.
\end{itemize}

It is more
surprising that the substitution $u=1/4$ will give us the answer we are 
seeking. 
% Recall that $p_n$ denotes the probability that two randomly
% selected elements of $\cal B_n$ are isomorphic.
Let $[z^n]g(z)$ denote the coefficient of $z^n$ in the power series $g(z)$.

\begin{lemma} \label{lnormalize}
For all positive integers $n\geq 2$, the probability~$p_n$
 that two phylogenetic trees
of size~$n$ are isomorphic satisfies
 \[p_n= \left (\frac{n!}{(2n-3)!!}\right)^2 \cdot [z^n]
  F\left(z,\frac14\right)\,.\]
\end{lemma}

\begin{proof}
Consider the sample space 
whose elements are the elements of 
$\cal U_n$, and
in which the probability of $t\in \cal U_n$ is
\begin{equation}\label{kdef}
\kappa(t):=\frac{n!}{2^{\sym(t)}}\cdot \frac{1}{(2n-3)!!}
=\frac{w(t)}{(2n-3)!!}.
\end{equation}
(For probabilists, $\kappa$ is the image on~$\cal U_n$ of the uniform
 distribution of~$\cal B_n$.)
For instance,
if $n=4$, then this space has two elements, (the two 
trees $L,R$ of~\eqref{utrees14}), one  has probability
$1/5$, and the other has probability $4/5$.
If we select two elements of this space at random, the probability that
they coincide is
\[p_n=\sum_{t\in \cal U_n} \kappa(t)^2 = 
\frac{1}{(2n-3)!!^2} \sum_{t\in \cal U_n} w(t)^2=
\frac{n!^2}{(2n-3)!!^2} \sum_{t\in \cal U_n}
 \left (\frac{1}{4}\right) ^{\sym(t)}.\]

Our claim now follows since $\sum_{t\in \cal U_n} 
\left (\frac{1}{4}\right) ^{\sym(t)}$ is indeed
the coefficient of $z^n$ in $F(z,1/4)$, in accordance with the definition~\eqref{deff}.
\end{proof}

\section{Isomorphism: Singularity Analysis} \label{ana-sec}
By Lemma~\ref{lnormalize},
 our goal is now to find the coefficient of $z^n$ in the one-variable
generating function 
\[
f(z):=F(z,1/4). 
\]
Lemma \ref{lemmafunc} shows that the formal power series
 $F(z,u)$ is the solution of the quadratic
equation (\ref{funceq}) that satisfies $F(0,0)=0$. 
That is, 
\begin{equation} \label{funcy}
F(z,u)=1-\sqrt{1-2z-(2u-1)F(z^2,u^2)}.\end{equation}
Iterated applications of
 (\ref{funcy}), starting with~$u=1/4$, show that
\[\renewcommand{\arraycolsep}{2truept}\renewcommand{\arraystretch}{1.8}
\begin{array}{lll}
f(z) ~\equiv~ F(z,1/4) &=& \ds
1-\sqrt{1-2z+\frac{1}{2}
F\left(z^2,\frac{1}{16}\right)}
\\[1.5truemm]
&=&
\ds
1-\sqrt{\frac{3}{2}-2z-\frac{1}{2}
\sqrt{1-2z^2+\frac{7}{8}F\left(z^4,\frac{1}{256}\right)}}~=~\cdots\,.
\end{array}
\]
In the limit, there results that $f(z)$ admits a 
``continued square-root'' expansion
% \[f(z)=F(z,1/4)=\]
\[f(z)=1-\sqrt{\frac{3}{2}-2z-\frac{1}{2}
\sqrt{\frac{15}{8}-2z^2-\frac{7}{8}
\sqrt{\frac{255}{128}-2z^4-\frac{127}{128}
% \sqrt{1-2z^8-F(z^{16},4^{-16})}}}}.\]
\sqrt{\cdots\vphantom{\frac{1}{1}}}}}},\]
out of which initial elements of the sequence $(p_n)_{n\ge1}$ are
 easily determined:
\[
1,~1,~1,~{\frac {17}{25}},~\frac{3}{7},~{\frac {5}{21}},~{\frac {13}{99}},~{\frac {
1385}{20449}},~{\frac {17861}{511225}},~{\frac {101965}{5909761}},~\cdots\,.
\]

In order to compute the growth rate of the coefficients of $f(z)$, we
will analyze the dominant singularity (or singularities) of this power
series.  The interested reader is invited to consult the book
\emph{Analytic Combinatorics} by Flajolet and Sedgewick \cite{FlSe08}
for more information on the notions and techniques that we are going
to use. Part of the difficulty of the problem is that the 
functional relation~\eqref{funcy} has the character of an 
inclusion--exclusion formula:
$F(z,u)$ does \emph{not} depend positively on
 $F(z^2,u^2)$, as soon as $u\le 1/2$,
which requires suitably crafted arguments, in contrast to the  
(simpler) asymptotic analysis of
$u_n=[z^n]F(z,1)$.

Briefly, we are interested in the {\em location}, {\em type}, and {\em number}
of the dominant singularities of $f(z)$, that is, singularities that have
smallest absolute value (modulus). 

\subsection{Location} 
First, it is essential for our analytic arguments to establish 
that $f(z)$ has a radius of convergence 
strictly less than~1. Our starting point  
parallels  Lemmas~1--2 of McKeon~\cite{McKeon96},
but we need a specific argument for the upper bound.

\begin{lemma}\label{rho-lem}
Let $\rho$
 be the largest real number such that $f(z)$ is analytic in the
interior of a disc centered at the origin that has radius $\rho$.
The following inequalities hold:
\[
0.4<\rho<0.625.
\]
\end{lemma}
\begin{proof}
$(i)$~\emph{Lower bound.}  
Note that $f(z)$ is convergent in some disc of radius \emph{at least} 0.4, 
since the coefficients of $f(z)=F(z,1/4)$ are at most as large as the 
coefficients of $F(z,1)$, the generating function $U(z)$ 
of Otter trees, and the
latter is known to be convergent in a disc of radius $0.40269\cdots$:
see Otter's original paper~\cite{Otter48} and
 Finch's book~\cite[\S5.6]{Finch03}
 for more details
on the asymptotics of $F(z,1)=U(z)$.

$(ii)$~\emph{Upper bound.}
%  on the radius of convergence of $F(z,1/4)$, 
For fixed $n$, let $a_1,a_2, \cdots, a_{u_n}$ be the numbers of
 our labeled trees whose
underlying unlabeled tree is the first, second, \ldots,
last Otter tree of size $n$.
Then  the relation
\begin{equation}
\label{compare}
p_n\equiv \frac{a_1^2+a_2^2+\cdots +a_{u_n}^2}{(a_1+a_2+\cdots +a_{u_n})^2} 
~>~\frac{1}{u_n},\end{equation}
results from the Cauchy-Schwarz inequality. (In words: the probability of coincidence
of two elements from a finite probability space is smallest when the
distribution is the uniform one.)

As we mentioned, it is proved in \cite{Otter48} that the generating function
$\sum_{n}u_nx^n$ converges in a disc of radius at least 0.4. 
% The interested
% reader can consult Section 5.6 of \cite{Finch03} for a brief, easy-to-read
% overview of this and related results.  
Therefore, the 
series $\sum_{n}\frac{1}{u_n}x^n$ converges in a disc of radius at
most $1/0.4=2.5$, and by (\ref{compare}), this implies that 
$\sum_n p_nx^n$ converges in a disc of radius less than 2.5. Now
Lemma \ref{lnormalize} shows that $F(z,1/4)$ is convergent in a disc of
radius less than $2.5/4=0.625$, since the coefficients of $F(z,1/4)$ are,
up to polynomial factors, $4^n$ times larger than the 
coefficients of $\sum_n p_nx^n$.
 It
follows
that $\rho <0.625$.
\end{proof}

A well-known theorem of Pringsheim states that if a function $g(z)$ is
representable around the origin by a
series expansion that has non-negative coefficients and radius of convergence
$R$, then the real number $R$ is actually a singularity of $g(z)$.
Applying this theorem to $f(z)$, we see that the positive real number
 $\rho$ must be
a singularity of $f(z)$.

\subsection{Type} \label{type-subsec}
Recall that a function $g(z)$ analytic in a domain~$\Omega$ is said to
have a \emph{square-root singularity} at a boundary point~$\alpha$ if,
for some function $H$ analytic at~0,
the representation $g(z)=H(\sqrt{z-\alpha})$ holds
in the intersection of~$\Omega$ and a neighborhood of~$\alpha$.
(In particular, if $g(z)=\sqrt{\gamma(z)}$ with $\gamma$ analytic at~$\alpha$,
then $g(z)$ has a square-root singularity at~$\alpha$ 
whenever $\gamma(\alpha)=0$
and $\gamma'(\alpha)\not=0$.)

\begin{lemma}\label{type-lemma}
All dominant singularities (of
 modulus $\rho$) of $f(z)$ are isolated and are of the square-root type.
\end{lemma}
\begin{proof} In order to see this, note that
$\rho<1$ (proved in Lemma~\ref{rho-lem})
implies that $\rho < \sqrt{\rho}$. Therefore, the power series
$F(z^2,1/4)$ (that has radius of convergence $\sqrt{\rho}$) 
is analytic in the interior of the disc of radius $\rho$,
and so is the power series $F(z^2,1/16)$ since its coefficients are smaller
than the corresponding coefficients of $F(z^2,1/4)$. Consequently,
Equation~(\ref{funcy}) implies that the dominant singularities of
\begin{equation}\label{toto}
f(z)=F\left(z,\frac14\right)=1-\sqrt{1-2z+
\frac{1}{2} F\left(z^2,\frac{1}{16}\right)}
\end{equation}
are of the square-root type: they are to be found amongst the roots of
the expression under the square-root sign in  (\ref{funcy}), that is, 
amongst the zeros of 
$1-2z+\frac{1}{2} F(z^2,1/16)$ that have modulus $\rho$. 
As $1-2z+\frac{1}{2} F(z^2,1/16)$ is analytic in 
the disc centered
at the origin with radius at least $\sqrt{\rho}>\rho$, it has isolated roots. 
Hence $f(z)$ has only a {\em finite} number of singularities on the circle
$|z|=\rho$, and each is of square-root type. 
\end{proof}

The argument of the proof (see~\eqref{toto}) also shows
 that $\rho$ is determined as
the smallest positive root of the equation
\begin{equation}\label{rhoeq}
1-2\rho+
\frac{1}{2} F\left(\rho^2,\frac{1}{16}\right)=0.
\end{equation}

\subsection{Number}
In order to complete our characterization of the dominant singular structure
of~$f(z)$, we need the following statement.

\begin{lemma} \label{number-lem}
The point $\rho$ is the {\em only} 
singularity of smallest modulus of $f(z)$. 
\end{lemma}
\begin{proof}
The argument is somewhat indirect and it proceeds in two stages.

First we
show that, as a power series, $f(z)$ converges for each $z$ with $|z|=\rho$.
To this purpose, we need to recall briefly 
some principles of singularity analysis,
 as expounded in~\cite[Ch.~VI]{FlSe08}.
Let $g(z)$ be a function analytic in~$|z|<R$ with finitely many singularities
at the set~$\{\alpha_j\}$ on the circle $|z|=R$;
assume in addition that $g(z)$ has a square-root singularity at
 each~$\alpha_j$
in the sense of Subsection~\ref{type-subsec}. Then,
one has $
[z^n]g(z) = O\left(R^{-n} n^{3/2}\right)$.
(This corresponds to the $O$--transfer theorem
 of~\cite[Th.~VI.3, p.~390]{FlSe08},
with amendments for the case of multiples singularities 
to be found in~\cite[\S VI.5]{FlSe08}; see also~\eqref{sap} below.)
It follows from this general estimate and Lemma~\ref{type-lemma}
that 
\[
[z^n]f(z)=O(\rho^{-n}n^{3/2}).
\]
Therefore, the series expansion of $f(z)$ converges absolutely
as long as $|z|\le \rho$, and, in particular, it converges for all $z$ with
modulus $\rho$. 

Now, we are in a position to prove that $f(z)$ has no singularity other
than $\rho$ on the circle $|z|=\rho$. 
Let us assume the contrary; that is, there is a real number $z_0\neq
\rho$ such that $|z_0|=\rho$ and $z_0$ is a singularity of
 $f(z)\equiv F(z,1/4)$.
Then, it follows from~(\ref{funcy}) that $f(z_0)\equiv F(z_0,1/4)=1$, 
since the expression
under the square-root sign in (\ref{funcy}) is equal to 0, corresponding to 
a singularity of square-root type. 
On the other hand, one has a priori $|f(z_0)|\le f(\rho)$,
as a consequence of the triangle inequality and the fact, proved above, that
$f(z)$ converges on~$|z|=\rho$. Now it follows from the \emph{strong
 triangle inequality}
that the equality $f(z_0)=f(\rho)$ is only possible if all the terms
$f_nz_0^n$ that compose the (convergent) series expansion of~$f(z_0)$
are positive real. (Here $f_m=[z^m]f(z)$.)
 However,  since, in particular, $f_1=1$ is nonzero, this
 implies that 
$z_0=\rho$, and a contradiction has been reached. 
(This part of the argument is also closely related to the Daffodil 
Lemma of~\cite[p.~266]{FlSe08}.)
\end{proof}

\subsection{The asymptotics of $p_n$}
As a result of Lemmas~\ref{rho-lem}--\ref{number-lem},
the function $f(z)$ has only one dominant singularity, and that
singularity $\rho$ is of the square-root type. 
One then has, for a family of constants~$h_k$, 
the local singular expansion:
\begin{equation}\label{locsing}
f(z)= 1+\sum_{k=0}^\infty h_k (1-z/\rho)^{k+1/2},
\end{equation}
which is valid for~$z$ near~$\rho$.
The conditions of the singularity analysis process as 
summarized in~\cite[\S VI.4]{FlSe08}
are then satisfied. Consequently,  each singular element
 of~\eqref{locsing} relative to~$f(z)$
can be translated into a matching asymptotic term relative to~$[z^n]f(z)$,
according to the rule
\begin{equation}\label{sap}
\sigma(z)=(1-z/\rho)^\theta \quad\longrightarrow\quad
[z^n]\sigma(z) = \rho^{-n} \binom{n-\theta-1}{n}\sim 
\rho^{-n}\frac{n^{-\theta-1}}{\Gamma(-\theta)}.
\end{equation}
In particular, we have $[z^n]f(z)\sim C\cdot \rho^{-n} n^{-3/2}$,
 for some~$C$.
 
Hence Lemma \ref{lnormalize}, combined  with Lemmas~\ref{type-lemma}--\ref{number-lem} 
and the routine asymptotics of
$n!/(2n-3)!!$ by Stirling's formula, leads to the following theorem. 

\begin{theorem} \label{thm1}
The probability that two phylogenetic trees of size~$n$ are
 isomorphic
admits a complete asymptotic expansion
\begin{equation}\label{pnasy}
p_n \sim a \cdot b^{-n} \cdot n^{3/2}\left(1+\sum_k \frac{c_k}{n^k} \right) ,
\end{equation}
where $a$, $b=4\rho$, and the $c_k$ are computable constants,
 with values
$a=3.17508\cdots$, $b= 2.35967\cdots$, and $c_1$ approximately equal to $-0.626$.
\end{theorem}

The function $F(z,u)$ can be determined numerically to great accuracy 
(by means of the recursion corresponding to the functional
 equation~\eqref{funcy}). So, the value
\[
\rho =  0.58991\,82714\,85535 \cdots, % 0308189461
\]
is obtained as the smallest
 positive root of~\eqref{rhoeq}; % $1-2\rho+\frac12 F(\rho^2,1/16)=0$;
the constant $a$ then similarly results from an evaluation of
 $F'\left(\rho^2,\frac{1}{16}\right)$;
the constant~$c_1$, which could in principle be computed in the same manner, was,
in our experiments, simply
estimated from the values of~$p_n$ for small~$n$. The formula~\eqref{pnasy},
truncated after its $c_1/n$ term,
then appears to approximate $p_n$ with a relative accuracy better than 
$10^{-2}$ for $n\ge5$, $10^{-4}$ for $n\ge38$, and $10^{-5}$ for $n\ge47$.

\section{Symmetrical Nodes and Automorphisms} \label{auto-sec}

\def\phi{\varphi}

In the course of our investigations on analytic properties of 
the bivariate generating function~$F(z,u)$, we came up with a few additional
estimates, which improve on those of McKeon~\cite{McKeon96}.
In essence, what is at stake is a perturbative analysis of~$F(z,u)$
and its associated singular expansions, for various values of~$u$,
in a way that refines  the developments of the previous section.
We offer here a succinct account: details can be easily supplemented
by referring to Chapter~IX of the book
 \emph{Analytic Combinatorics}~\cite{FlSe08}.

\begin{theorem} \label{thm2}
$(i)$~Let~$X_n$ be the random variable representing the number of symmetrical nodes
in a random  Otter tree of~$\cal U_n$. Then, $X_n$ satisfies a limit  law 
of Gaussian type,
\[
\forall x\in\R~:
\qquad
\lim_{n\to\infty}\Pr\left(X_n\le \mu n + \sigma x \sqrt{n}\right)
=\frac{1}{\sqrt{2\pi}}
\int_{-\infty}^x e^{-w^2/2}\, dw, 
\]
for some positive constants~$\mu$ and $\sigma$. Numerically, $\mu=0.35869\cdots$\,.

$(ii)$~Let~$Y_n$ be the random variable representing the number of symmetrical nodes
in a random  phylogenetic tree of~$\cal B_n$. Then, $Y_n$ satisfies a limit  law 
of Gaussian type,
\[
\forall x\in\R~:
\qquad\lim_{n\to\infty}\Pr\left(Y_n\le \wh\mu n + \wh\sigma x \sqrt{n}\right)
=\frac{1}{\sqrt{2\pi}}
\int_{-\infty}^x e^{-w^2/2}\, dw,
\]
for some positive constants~$\wh\mu$ and $\wh\sigma$. Numerically, $\wh\mu=0.27104\cdots$\,.
% Accordingly, the size of the automorphism
% group of  a random tree in~$\cal U_n$ satisfies 
% asymptotically a log-normal distribution.
\end{theorem}

\begin{figure}\small
\begin{center}
\Img{5.5}{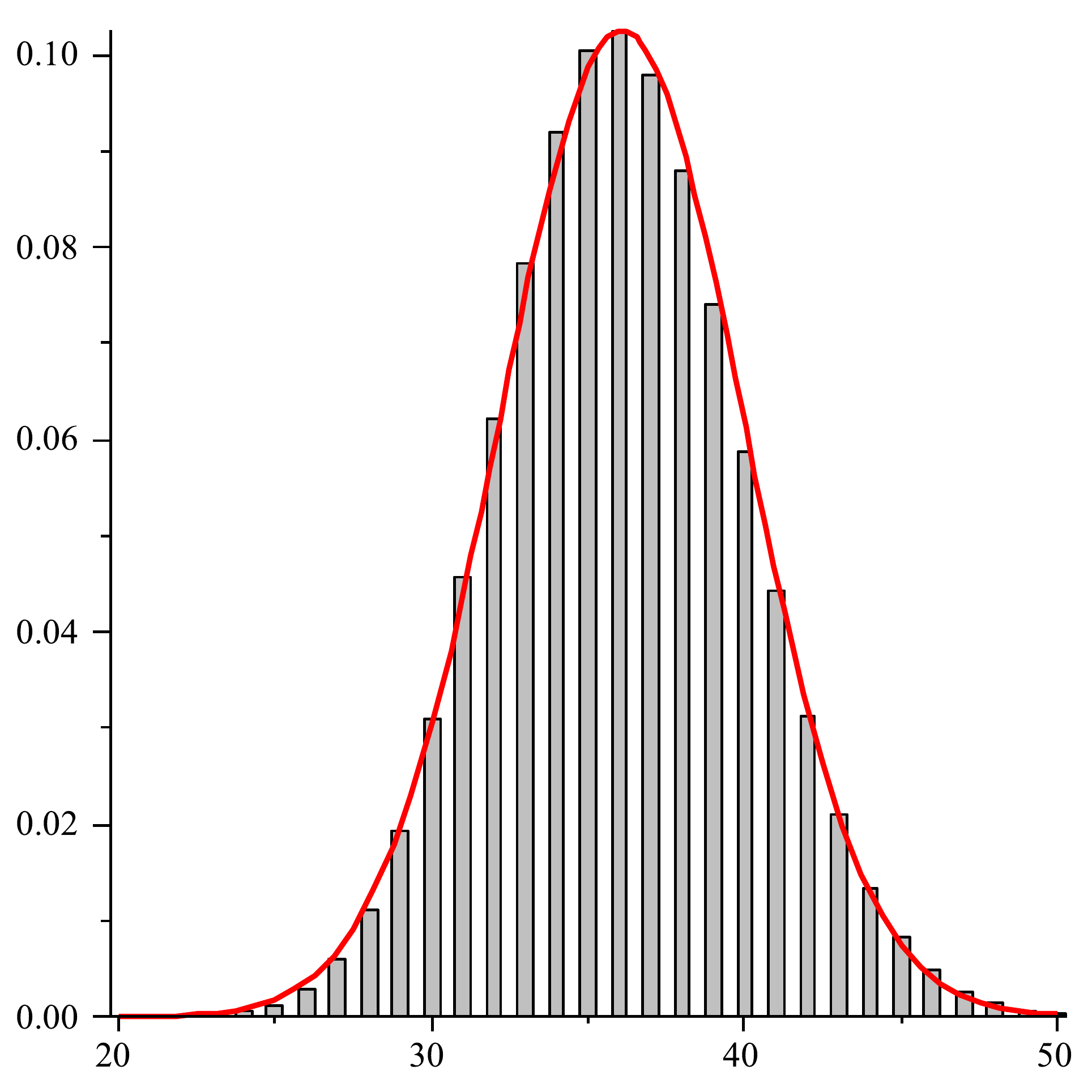}\qquad \Img{5.5}{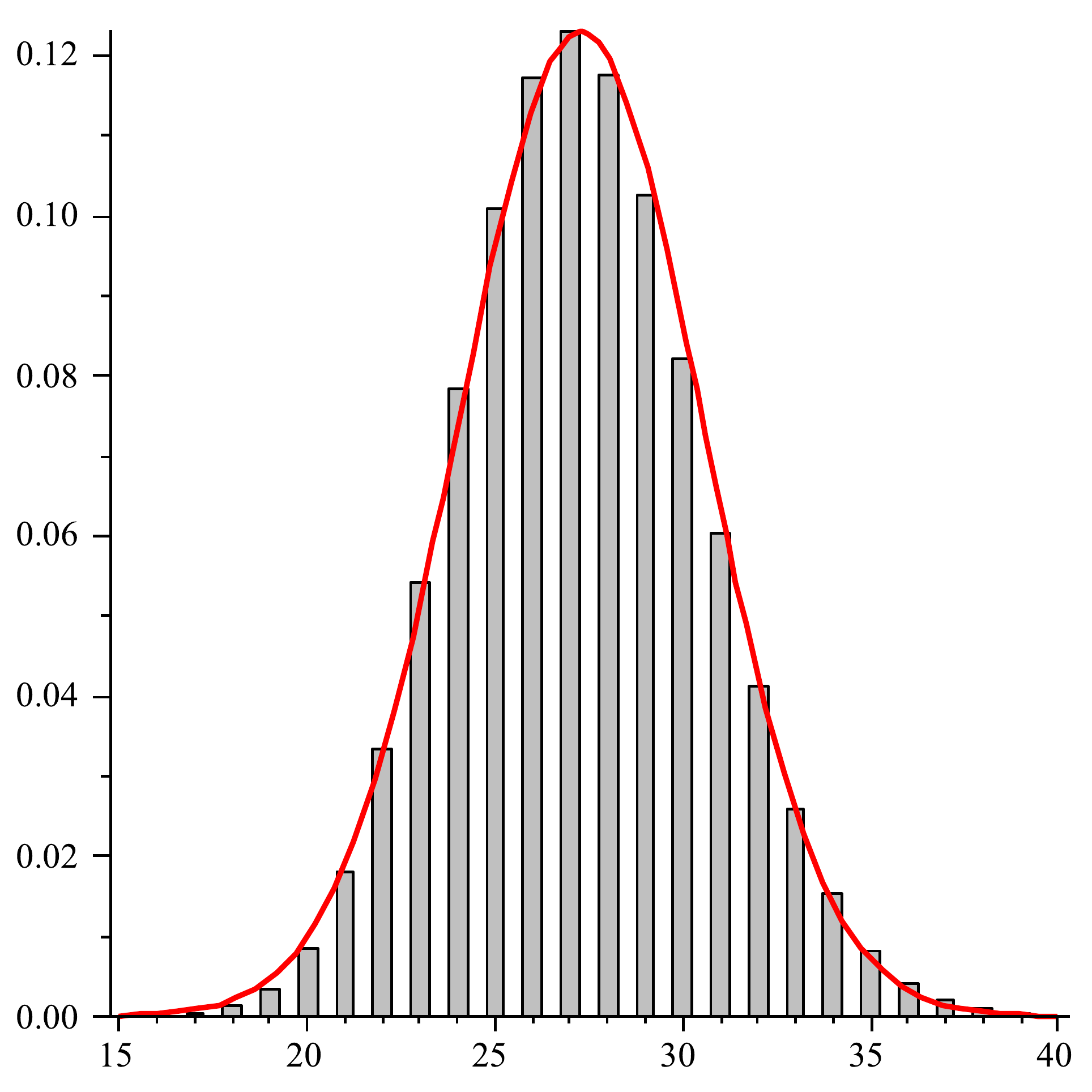}
\end{center}
\caption{\label{gauss-fig}\small
Histograms of the distribution of the number of symmetrical nodes in 
trees of size~$100$, compared to a matching Gaussian.
Left: Otter trees of 
$\cal U_{100}$. Right: phylogenetic trees of $\cal B_{100}$.}
\end{figure}

\begin{proof}[Proof (Sketch)]
$(i)$~\emph{The case of Otter trees $(X_n,\cal U_n)$}.
In accordance, with general principles~\cite[Ch.~IX]{FlSe08},
we need to estimate the generating polynomial
\begin{equation}\label{defphin}
\phi_n(u):=[z^n]F(z,u),
\end{equation}
when $u$ is close to~1, with~$F(z,u)$ as specified by~\eqref{deff} 
and~\eqref{funceq}.
For~$u$ in a small enough complex neighborhood~$\Omega$ of~$1$, the radius of
convergence of~$F(z^2,u^2)$ is larger than some $\rho_2>\rho_1$,
where $\rho_1\approx 0.40269$ is the radius of convergence associated
with Otter trees.
Then, by an argument similar to the ones used earlier,
there exists a solution~$\rho(u)$ to the analytic equation
\begin{equation}\label{ananeq}
1-2\rho(u)+(u-1)F(\rho(u)^2,u^2)=0
\end{equation}
(compare with~\eqref{rhoeq}), such that $\rho(1)=\rho_1$ is the
dominant singularity of the generating function~$F(z,1)$ of Otter trees.
By the analytic version of the implicit function theorem (equivalently,
by the Weierstrass Preparation Theorem), this function~$\rho(u)$
depends analytically on~$u$, for~$u$ near~$1$.

In addition, by~\eqref{funcy},
the function $F(z,u)$ has a singularity of the square-root type at~$\rho(u)$.
Also, for~$u\in\Omega$ and~$\Omega$ taken small enough,
the triangle inequality combined with the previously established properties of
$F(z,1)$ may be used to verify that there are no other singularities of $z\mapsto F(z,u)$ on $|z|=|\rho(u)|$.
There results, from singularity analysis
\emph{and} the uniformity of the  process~\cite[p.~668]{FlSe08},
the asymptotic estimate
\begin{equation}\label{phinasy}
\phi_n(u)={c(u)}\rho(u)^{-n}n^{-3/2}\left(1+o(1)\right), \qquad n\to+\infty,
\end{equation}
\emph{uniformly} with respect to~$u\in\Omega$, for some~$c(u)$ that is analytic at~$u=1$.
Then, the probability generating  
function of~$X_n$, which equals $\phi_n(u)/\phi_n(1)$ satisfies what is known as a
\emph{ ``quasi-powers approximation}.
That is, it resembles (analytically) the probability generating function
 of a sum of independent
random variables,
\begin{equation}\label{pgfUn}
\frac{\phi_n(u)}{\phi_n(1)}=\frac{c(u)}{c(1)} \left(\frac{\rho(1)}{\rho(u)}\right)^{n}
\left[1+\varepsilon_n(u)\right],
\end{equation}
where $\sup_{u\in\Omega}|\varepsilon_n(u)|$ tends to~0 as $n\to\infty$.
The Quasi-powers Theorem (see~\cite[\S IX.5]{FlSe08} and~\cite{Hwang98b})
precisely applies to such approximations by quasi-powers and implies that
the distribution of~$X_n$ is asymptotically normal. 

\smallskip
$(ii)$~\emph{The case of phylogenetic trees $(Y_n,\cal B_n)$}. 
The starting point is a simple combinatorial property of~$\phi_n(u)$, as defined in~\eqref{defphin}:
\begin{equation}\label{defphi2n}
\phi_n(u/2)=\frac{1}{n!} \sum_{t\in \cal U_n} \frac{n!}{2^{\sym(t)}} u^{\sym(t)}
=\frac{1}{n!} \sum_{t\in\cal B_n} u^{\sym(t)}.
\end{equation}
(The first form results from the definition~\eqref{deff} of $F(z,u)$;
the second form relies on the expression~\eqref{wdef} of the number of different labellings 
of an Otter tree that give rise to a phylogenetic tree.) Thus,
$\phi_n$ taken with an argument near~$1/2$ serves, up to normalization, as 
the probability generating function of the number of symmetrical nodes in phylogenetic trees
of~$\cal B_n$.

From this point on, the analysis of symmetries in phylogenetic trees is entirely similar to that of
Otter trees. For $u$ in a small complex neighborhood~$\wh\Omega$ of~$1/2$,
the generating function $z\mapsto F(z,u)$ has a dominant singularity~$\rho(u)$ 
that is an analytic solution of~\eqref{ananeq} and is such that $\rho(1/2)=1/2$,
the radius of convergence of $B(z)\equiv F(z,1/2)$. As a consequence, estimates that parallel
those of~\eqref{phinasy} and~\eqref{pgfUn} are seen to hold, but with $u\in\wh\Omega$ now near $1/2$. 
In particular,
\begin{equation}\label{pgfBn}
\frac{\phi_n(u)}{\phi_n(1/2)}=\frac{\wh c(u)}{\wh c(1/2)} \left(\frac{\wh \rho(1/2)}{\wh \rho(u)}\right)^{n}
\left[1+\wh \varepsilon_n(u)\right],
\end{equation}
where $\wh \varepsilon_n(u)\to0$ uniformly.
By the Quasi-powers Theorem (set $u:=v/2$, with~$v$ near~$1$), 
the distribution of~$Y_n$ is asymptotically normal.
\end{proof}

Figure~\ref{gauss-fig} shows that the fit with a Gaussian is quite good, even for 
comparatively low sizes ($n=100$).
Phrased differently, the statement of Theorem~\ref{thm2} means that the \emph{logarithm of
 the order~$2^{\sym(t)}$ of the automorphism group of
a random tree~$t$ (either in~$\cal U_n$ or in~$\cal B_n$) is normally distributed}\footnote{%
	The situation is loosely evocative of the fact (Erd\H os--Tur\'an Theorem)
	that the logarithm of the order of 
	a random permutation of size~$n$ is normally distributed; see, e.g.,~\cite{ErTu67,GoSc91b,Nicolas85}. 
}. 
In the case of~$\cal U_n$, the \emph{expectation}
of the cardinality of this group has been determined by
 McKeon~\cite{McKeon96} to
grow roughly as~$1.33609^n$. 
In the case of phylogenetic trees ($\cal B_n$), we find an \emph{expected} growth of the rough form 
$1.24162^n$, where the exponential rate $1.24162\cdots$ is exactly $1/(2\rho_1)$, with~$\rho_1$,
still, the radius of convergence of $U(z)\equiv F(z,1)$.
(These values are consistent with the fact that trees with a higher number of symmetries 
admit a smaller number of labellings, hence  are less likely 
to appear as ``shapes'', under the phylogenetic model~$\cal B_n$.)

As a matter of fact, the histograms of Figure~\ref{gauss-fig}
suggest that a convergence stronger than a plain convergence in law (corresponding to 
convergence of the distribution function) holds.
\begin{definition}\label{lll-def}
Let~$(\xi_n)$ be a family of random variables with expectation~$\mu_n=\E(\xi_n)$
and variance~$\sigma_n^2=\Var(\xi_n)$. 
It is said to satisfy a local limit law with density~$g(x)$ if one has
\begin{equation}\label{lll0}
\lim_{n\to\infty} \sup_{x\in\R}
\left|\sigma_n\Pr(\xi_n=\lfloor \mu_n+x\sigma_n\rfloor )-g(x)\right|=0.
\end{equation}
\end{definition}

\noindent
In other terms, we expect the probability of~$\xi_n$ being
at~$x$ standard deviations away from its mean to be well approximated by
$g(x)/\sigma_n$. 
This concept is discussed in the case of sums of random variables by Gnedeneko
and Kolmogorov in~\cite[Ch.~9]{GnKo68} and, in a broader combinatorial context,
by Bender~\cite{Bender73} and Flajolet--Sedgewick~\cite[\S IX.9]{FlSe08}.

\begin{theorem}\label{thm3}
The number of symmetrical nodes in either an unlabeled tree 
($X_n$ on~$\cal U_n$) or a phylogenetic tree ($Y_n$ on~$\cal B_n$)
satisfies a \emph{local limit law} of the Gaussian type.
That is, in the sense of Definition~\ref{lll-def},
a local limit law holds, with  density
\[
g(x)=\frac{1}{\sqrt{2\pi}}e^{-x^2/2}.
\]
\end{theorem}

\begin{proof}
$(i)$~\emph{The unlabeled case $(X_n,\cal U_n)$}.
The proof essentially boils down to establishing that
\[
f_n(u)=[z^n] F(z,u)
\]
is small compared to $[z^n]F(z,1)$, as soon as $u$ 
satisfies $|u|=1$ and stays away from~1; then,
Theorem~IX.14, p.~696, from~[FlSe08] does the rest. The arguments are
variations of the ones previously used.

Since a tree of size~$n$ has less than~$n$ symmetrical nodes,
we have $|f_n(u)|\le |u|^n f_n(1)$ for any $|u|\ge1$. There results 
that the convergence of the series expansion of~$F(z,u)$ is dominated by that of $F(|zu|,1)$, whenever $|u|\ge1$. 
Apply the fact explained in the previous sentence,
 with $z^2$ and $u^2$ instead of
$z$ and $u$, to get that the coefficients of 
$F(z^2,u^2)$ are less than the 
coefficients of $F(|z^2u^2|,1$),
where the latter series is convergent if
$|z^2u^2|<0.625$, or in other words, $|zu|<0.75$, say. 
Now choose $\eta$ so that
$(1+\eta) (\rho_1+\eta) < 0.75$,
where~$\rho_1$ is the radius of convergence of Otter trees ($\rho_1\equiv\rho(1)\approx0.40269$). Then 
$F(z^2,u^2)$ is bivariate analytic whenever $|z|<(\rho_1+\eta)$ and $|u|<1+\eta$.
% the function $F(z^2,u^2)$ is bivariate analytic
% in $|z|\le \rho_1+\eta$, $|u|\le 1+\eta$,
% where~$\rho_1$ is the radius of convergence of Otter trees ($\rho_1\approx0.40269$).
In accordance with previously developed arguments, this implies that, for 
any fixed~$u$ satisfying $|u|\le 1+\eta$, the function
$z\mapsto F(z,u)$ has only finitely many singularities, each of the
square-root type, in $|z|\le \rho_1+\eta$.

For~$u$ in a small complex neighborhood of~1, we already know that $z\mapsto F(z,u)$ has only \emph{one} dominant singularity at some
$\rho(u)$, which is a root of
\[
1-2\rho(u)+(2u-1)F(\rho(u)^2,u^2)=0.
\]
(This  property  lies at the basis  of  the central  limit  law of the
previous theorem.) 

Consider      now a~$u$ such that  $|u|=1$,    but
$u\not\in\Omega$.  We argue that  $z\mapsto F(z,u)$ is analytic at all
points~$z$ such that $|z|=\rho_1$. Indeed for such values of~$u$ and~$z$, 
we have, by the \emph{strong triangle inequality},
% (see also the Daffodil Lemma~\cite[p.~266]{FlSe08}
%  applied to~$F(z,u)$, viewed now as a function of~$u$),
\begin{equation}\label{ineq0}
|F(z,u)|<F(\rho_1,1),
\end{equation}
the reason being that, in the expansion $F(z,u)=z+uz^2+uz^3+\cdots$,
the values of the monomials $u^kz^n$ cannot be all collinear, unless~$u=1$.
The inequality~\eqref{ineq0} combined with the fact that $F(\rho_1,1)=1$
implies that $z\mapsto F(z,u)$ cannot be singular (since,
as we know, the only possibility for a singularity would be
that it is of the square-root type \emph{and}~$F(z,u)=1$).

Thus, for $|u|=1$ and~$u\not\in\Omega$, the function $z\mapsto F(z,u)$
is analytic at all points of~$|z|=\rho_1$. Hence, 
it is analytic in $|z|\le\rho_1+\delta$,
for some~$\delta>0$. By usual exponential bounds, there results that, for some~$K>0$,
one has
\begin{equation}\label{ineq1}
\left|f_n(u)\right|< K \left(\rho_1+\delta/2\right)^{-n},
\qquad
|u|=1,\quad u\not\in\Omega.
\end{equation}
As expressed by Theorem~IX.14\footnote{
        The reasoning corresponding to that theorem is simple: start from
\[
[u^k]f_n(u) = \frac{1}{2i\pi} \int_{|u|=1} f_n(u)\, \frac{du}{u^{k+1}}.
\]
Use~\eqref{ineq1} to neglect the contribution
corresponding to~$u\not\in\Omega$; appeal to the saddle point method
applied to the quasi-powers approximation to estimate the central part~$u\in\Omega$, 
and conclude.
} of~\cite{FlSe08}, the existence of a quasi-powers approximation 
(when~$u$ is near 1), as in~\eqref{phinasy} and~\eqref{pgfUn}, and of the exponentially small bound 
(when $u\not\in\Omega$ is away from~1), as provided by~\eqref{ineq1},
suffices to ensure the existence of a local limit law. 

$(ii)$~\emph{The labeled case $(Y_n,\cal B_n)$}.
In accordance with~\eqref{defphi2n},
the function $F(z,u/2)$ is the bivariate exponential generating function of
phylogenetic trees, with~$z$ marking size and~$u$ marking the number of
symmetrical nodes. Consider once more $|u|=1$
and distinguish the two cases $u\in\wh\Omega$ (for which the proof of Theorem~\ref{thm2}
provides a quasi-powers approximation) and $u\not\in\wh\Omega$.
In the latter case, arguments that entirely parallel those applied
to unlabeled trees give us that $z\mapsto F(z,u/2)$ has no singularity on~$|z|=1/2$.
This implies, for $u\not\in\wh\Omega$, the exponential smallness of $\wh\varphi_n(u/2)$, 
as defined in~\eqref{defphi2n},
resulting in an estimate that parallels~\eqref{ineq1}. 
Theorem~IX.14 of~\cite{FlSe08} again enables us to conclude as to the 
existence of a local limit law.
\end{proof}

\section{Coincidence of the Number of Symmetries}\label{same-sec}

From a statistician's point of view, it may be of interest to 
determine the probability for two trees to be ``\emph{similar}'' (rather
 than plainly isomorphic),
given some structural similarity distance between non-plane trees---see, for
 instance,
the work of Ycart and Van Cutsem~\cite{VaYc98} for a study conducted under
 probabilistic assumptions that 
differ from ours.
Combinatorial generating functions can still be useful in this broad range of problems,
as we now show by considering the following question:
\emph{determine the probability that two randomly chosen trees $\tau,\tau'$ 
of the same size have the same 
number of symmetrical nodes}. This probability \emph{a priori} lies in the interval~$[\frac1n,1]$;
we shall see, in Theorem~\ref{thm4}, that its asymptotic value is ``in-between''.

The problem under consideration belongs to an orbit of questions occasionally
touched upon in the literature. For instance, Wilf~\cite{Wilf05} showed 
that the probability that two permutations of size~$n$ have the same number
of cycles is asymptotic to $(2\sqrt{\pi \log n})^{-1}$; B\'ona and Knopfmacher~\cite{BoKn08}
examine combinatorially and asymptotically 
the probability that various types of integer compositions have the same number of parts,
and several other coincidence probabilities are studied in~\cite{FlFuGoPaPo06}.
The following basic lemma  trivializes the asymptotic side of several such questions.

\begin{lemma}\label{coinc-lem}
Let $\cal C$ be a combinatorial class equipped with an integer-valued para\-meter~$\chi$.
Assume that the random variable corresponding to~$\chi$ restricted to~$\cal C_n$ 
(under the uniform distribution over~$\cal C_n$) satisfies a local limit law with density~$g(x)$,
in the sense of Definition~\ref{lll-def}. 
Let the variance of~$\chi$ on~$\cal C_n$ be~$\sigma_n^2$ and assume that $g(x)$ is 
continuously differentiable.
Then, the probability that two objects $c,c'\in\cal C_n$ 
admit the same value of~$\chi$ satisfies the asymptotic estimate
\begin{equation}\label{coinc0}
\Pr\bigg[\chi(c)=\chi(c'),\quad c,c'\in\cal C_n\bigg]
\sim 
\frac{K}{\sigma_n},
\qquad\hbox{where}\quad 
K:=\int_{-\infty}^{\infty} g(x)^2\, dx.
\end{equation}
\end{lemma}

\noindent
Note that, for $g(x)$ the standard Gaussian density, one has $K=1/(2\sqrt{\pi})$.

\begin{proof}[Proof (sketch)]
Let~$\varpi_n$ be the probability of coincidence; that is, the left hand-side of~\eqref{coinc0}.
Observe that, by hypothesis, we must have $\sigma_n\to\infty$.
The baseline  is that
\[\renewcommand{\arraystretch}{1.8}
\begin{array}{lllll}
\varpi_n & =& \ds \sum_k \Pr_{\cal C_n}[\chi(c)=k]^2 \\
&\sim & \ds \frac{1}{\sigma_n^2}\sum _{x\in \cal E_n} g(x)^2, && \ds 
\hbox{with}\quad \cal E_n:=\frac{1}{\sigma_n} \left(\cal Z_{\ge0}-\{\mu_n\}\right),
\quad \mu_n:=\E_{\cal C_n}[\chi]\\
&\sim&\ds  \frac {1}{\sigma_n}\int_{-\infty}^{\infty} g(x)^2\, dx.
\end{array}
\]
To justify this chain rigorously, first restrict attention to values of~$x$ in a finite interval 
$[-A,+B]$, so that the tails $(\int_{<A}+\int_{>B}) g$ are less than some
small~$\epsilon$.
Then, with $x\in[-A,+B]$, make use of the approximation~\eqref{lll0} provided by
the assumption of a local limit law. Next, approximate the sum of $g(x)^2$ taken at 
regularly spaced  sampling points (a Riemann sum) by the corresponding integral. Finally, complete back the tails.
\end{proof}

Given the local limit law expressed by Theorem~\ref{thm3}, an immediate consequence
of Lemma~\ref{coinc-lem} is the following.

\begin{theorem}\label{thm4}
For Otter trees $(\cal U_n)$ and phylogenetic trees $(\cal B_n)$,
the asymptotic  probabilities that two trees of size~$n$ have the same number of symmetries
admit the forms
\[
{\cal U}_n~: \quad \frac{1}{2\sigma \sqrt{\pi n}},
\qquad
{\cal B}_n~: \quad  \frac{1}{2\wh \sigma \sqrt{\pi n}},
\]
where $\sigma,\wh\sigma$ are the two ``variance constants'' of Theorem~\ref{thm2}.
\end{theorem}

In summary, as we see in 
several particular cases here, \emph{qualitatively} similar phenomena are expected in trees,
whether plane or non-plane trees,  labelled or unlabelled, whereas, \emph{quantitatively},
the structure constants (for instance, $\mu$ and~$\wh \mu$ in Theorem~\ref{thm2};
$\sigma$ and $\wh\sigma$ in Theorem~\ref{thm4})
tend to be model-specific. Yet another instance of 
such universality  phenomena is the height of Otter trees, analysed in~\cite{BrFl08},
which is to be compared to the height of plane binary trees~\cite{FlOd82}:
both scale to $\sqrt{n}$ and lead to the same elliptic-theta distribution, albeit with different scaling factors.

\smallskip
\begin{small}%
\noindent
{\bf Acknowledgements.}
The work of M.~B\'ona
was partially supported by the National Science Foundation and the
National Security Agency.
The work of P. Flajolet was partly supported by the French ANR Project SADA (``Structures Discr\`etes
et Algorithmes'').\par
\end{small}

\def\cprime{$'$}

% \bibliographystyle{acm}
% \bibliography{algo}
%\end{document}

% \begin{thebibliography}{99}
% \bibitem{Finch03} S. Finch, Mathematical Constants, Encyclopedia of Mathematics
% and Its Applications {\bf 94}, Cambridge University Press, 2003. 
% \bibitem{FlSe08} P. Flajolet, R. Sedgewick, Analytic Combinatorics,
% Cambridge University Press, to appear.
% \bibitem{McKeon91} K. A. McKeon, The expected number of symmetries in 
% locally-restricted trees I, pp. 849-860 of Y. Alavi et al., eds.,
%  Graph Theory, Combinatorics, and Applications. Wiley, NY, 2 vols., 1991.
% \bibitem{Otter48} R. Otter, The number of trees, {\em Annals of Math.} {\bf 49}
% (1948), 583-599.
% \bibitem{Sloane08} N. J. A. Sloane,  The On-Line Encyclopedia of Integer 
% Sequences, {\tt www.research.att.com/~njas/sequences}. 
% \bibitem{Stanley99} R. Stanley, Enumerative Combinatorics, Volume 2, 
% Cambridge University Press, 1999.
% \end{thebibliography}
\end{document}